\documentclass[11pt]{amsart}
\usepackage{amsmath,mathtools}
\usepackage{amsmath, amsthm, amssymb, amsfonts, enumerate}

\usepackage[colorlinks=true,linkcolor=blue,urlcolor=blue]{hyperref}
\usepackage{dsfont}
\usepackage{color}
\usepackage{geometry}
\usepackage{todonotes}
\usepackage{epstopdf}
\newcommand{\stkout}[1]{\ifmmode\text{\sout{\ensuremath{#1}}}\else\sout{#1}\fi}

\DeclareMathOperator*{\Tr}{Tr}
\DeclareMathOperator*{\range}{R}

\DeclareMathOperator*{\Span}{Span}

\newtheorem{theorem}{Theorem}[section]
\newtheorem{remark}[theorem]{Remark}

\newtheorem{definition}[theorem]{Definition}

\def \E{\mathsf{E}}

\def \P{\mathbb{P}}
\def \R{\mathbb{R}}

\def\d{\mathrm{d}}

\definecolor{red}{rgb}{1.0,0.0,0.0}

\definecolor{blu}{rgb}{0.0,0.0,1.0}

\definecolor{gre}{rgb}{0.03,0.50,0.03}

\def \epsilon{\varepsilon}

\title[Partial Regularity of  Semiconvex Viscosity Supersolutions to Nonlinear PDEs]{Partial Regularity of Semiconvex Viscosity Supersolutions to {Fully nonlinear} elliptic HJB equations and Applications to Stochastic Control} 
\author[Federico]{Salvatore Federico}
\author[Ferrari]{Giorgio Ferrari}
\author[Rosestolato]{Mauro Rosestolato}

\address{S.~Federico: Dipartimento di Matematica, Universit\`a di Bologna,  Piazza di Porta S.\ Donato 5, 40126, Bologna, Italy}
\email{\href{mailto:s.federico@unibo.it}{s..federico@unibo.it}}
\address{G.~Ferrari: Center for Mathematical Economics (IMW), Bielefeld University, Universit\"atsstrasse 25, 33615, Bielefeld, Germany}
\email{\href{mailto:giorgio.ferrari@uni-bielefeld.de}{giorgio.ferrari@uni-bielefeld.de}}
\address{M.~Rosestolato: Dipartimento di Economia, Universit\`a di Genova, Via F.\ Vivaldi 5, 16126, Genova, Italy}
\email{\href{mailto:mauro.rosestolato@unige.it}{mauro.rosestolato@unige.it}}

\date{}

\numberwithin{equation}{section}

\begin{document}

\begin{abstract}
  In this note, we show that a locally semiconvex viscosity supersolution of a possibly degenerate fully nonlinear elliptic Hamilton–Jacobi–Bellman (HJB) equation is differentiable along the directions spanned by the range of the coefficient matrix associated with the second-order term. The proof combines tools from convex analysis with a contradiction argument. This result has significant implications for stationary stochastic control problems. In drift-control problems, it enables the construction of a candidate optimal feedback control in the classical sense and the establishment of a verification theorem. Moreover, in optimal stopping and impulse control problems, when the second-order term is nondegenerate, the value function is shown to be continuously differentiable.
\end{abstract}



\maketitle

\section{Introduction}
\label{introduction}
Let ${A}$ be a nonempty set, $\mathcal{O} \subseteq \mathbb{R}^n$ open, 
$\sigma\colon \mathcal{O} \times {A} \to \mathbb{R}^{n \times m}$, $\beta\colon \mathcal{O} \times {A} \to \mathbb{R}^n$, 
$g, \rho\colon \mathcal{O} \times {A} \to \mathbb{R}$, and define 
\begin{equation}\label{2024-09-22:00}
\mathcal{L}(a,x,r,p,P) :=
g(x,a) + \langle \beta(x,a), p \rangle + \frac{1}{2} \Tr \left( \sigma(x,a)\sigma^*(x,a) P \right) - \rho(x,a)r,
\end{equation}
and
\begin{equation}\label{eq:2024-10-08:01}
\mathcal{H}(x,r,p,P) := \sup_{a \in {A}} \mathcal{L}(a,x,r,p,P), \quad x \in \mathcal{O}, \ r \in \mathbb{R}, \ p \in \mathbb{R}^n, \ P \in \mathcal{S}_n,
\end{equation}
where $\mathcal{S}_n$ denotes the set of real symmetric $n \times n$ matrices.
We consider the stationary Hamilton-Jacobi-Bellman (HJB) equation
\begin{equation}\label{2024-07-28:12}
\mathcal{H}(x, v(x), Dv(x), D^2v(x)) = 0.
\end{equation}
Given a locally semiconvex viscosity supersolution $v$ to this (\emph{possibly degenerate}) fully nonlinear elliptic PDE, we prove in Theorem \ref{prop:main} that $v$ is continuously differentiable along any direction in the range of the coefficient $\sigma(x,a)$ as $a$ varies in the set $A$.

Regularity of semiconvex viscosity supersolutions for a class of fully nonlinear, nondegenerate PDEs has been studied in \cite{BM}. Here, we focus on \emph{partial} regularity for \emph{possibly degenerate} HJB-type PDEs as described above.

Our results have immediate consequences for stationary, discounted stochastic optimal control problems for multi-dimensional Itô diffusions $X$, including standard stochastic control problems, as well as optimal stopping and impulse control problems. Indeed, the dynamic programming principle implies that the value function $V$ in these problems is a viscosity (super)solution of the corresponding dynamic programming equation, and hence a viscosity supersolution of an equation such as \eqref{2024-07-28:12}. Moreover, semiconvexity is a well-known property in (stochastic) optimal control theory (see \cite{CannarsaSinestrari} and \cite[Chapt.\ 4.4]{YongZhou} for the regular control case).

To the best of our knowledge, this is the first paper to link semiconvexity and the viscosity supersolution property in order to establish \emph{partial regularity} of the value function for regular control, impulse control, and optimal stopping problems in a general setting. This partial regularity is particularly relevant for regular control problems in which the control acts only on the drift of the system. It also yields a directional \emph{smooth-fit principle}, which is useful for characterizing the optimal exercise boundary in stopping and impulse control problems. We discuss these features in more detail below.

In regular stochastic control problems, the decision maker adjusts the drift and/or diffusion coefficients of a stochastic differential equation (SDE) to optimize a performance criterion depending on the state variable $X$ and the control process. Semiconvexity (or semiconcavity) of the value function is well-established when the problem data possess such properties (see \cite{CannarsaSinestrari} for the deterministic case and \cite[Prop.\ 4.5]{YongZhou} for the stochastic case). Our result applies directly in this context, allowing the construction of a candidate optimal control in feedback form and the establishment of a verification theorem under the viscosity solution framework (see \cite[Ch.5, Sec.2]{YongZhou}). Additionally, under hypoellipticity assumptions on the linear part of the second-order operator, this partial regularity may lead to classical solutions of the HJB equation, providing a stronger foundation for the verification theorem.

In optimal stopping problems, where the goal is to choose an optimal stopping time to maximize a performance criterion, the value function $V$ typically satisfies a variational inequality. Regularity properties such as the \emph{smooth-fit principle} are crucial for understanding the structure of the continuation and stopping regions (see \cite{peskir2006optimal} for an exposition). Sobolev regularity of viscosity solutions to fully nonlinear (parabolic and elliptic) obstacle problems has been studied under uniform ellipticity by \cite{Byun18}, \cite{Byun22}, and \cite{Krylov} in bounded smooth domains, and more recently by \cite{Durandard} for irregular obstacles and non-smooth domains. Our findings contribute by showing $C^1$-differentiability of semiconvex viscosity supersolutions to obstacle problems in relevant directions, even in degenerate cases. In the nondegenerate setting, applied to optimal stopping, this implies that it is never optimal to stop at points of non-differentiability of the obstacle.

For impulse control problems, in which the decision maker determines the timing and magnitude of jumps in the state process (see \cite{BeLioIC} and \cite{OksSul} for analytical and probabilistic expositions, respectively), our approach provides a simple proof that semiconvexity and the viscosity supersolution property ensure directional differentiability of the value function. This facilitates the characterization of optimal policies even in degenerate cases, where regularity results are typically more difficult to obtain.

The rest of the paper is organized as follows. Section \ref{sec:setting} introduces the setting and the main result. Section \ref{sec:drift-control} explores applications to drift-control problems, while Section \ref{sec:OSIC} addresses optimal stopping and impulse control problems.


\subsection*{Notation} Throughout the paper, the set $\mathbb{N}$ denotes the set of natural numbers $\mathbb{N} = \{1, 2,\dots\}$, we let $n, m \in \mathbb{N}$, and we use the notation $\R^n_{+}\coloneqq (0,\infty)^n$. 

We denote by $B_{R}(x_{0})$ the open ball of radius $R>0$ centered at $x_{0}\in \R^{n}$. We write $\mathcal{O} \subset \R^n$ for an open set and $C(\mathcal{O}; \R^n)$ (respectively, by $C(\mathcal{O}; \R^{n\times m}))$ for the space of continuous functions from $\mathcal{O}$ to $\R^n$ (respectively, from $\mathcal{O}$ to $\R^{n\times m}$). Furthermore, for $k\in \mathbb{N}$, we let $C^k(\mathcal{O}; \R^n)$ (respectively, $C^k(\mathcal{O}; \R^{n\times m}))$ be the space of functions from $\mathcal{O}$ to $\R^n$ (respectively, from $\mathcal{O}$ to $\R^{n\times m}$) that are $k$ times continuously differentiable.  The gradient and the Hessian of $C^2$ functions are denoted by $D$ and $D^2$, respectively.

We write $\mathcal{L}(\R^n)$ for the set of linear operators on $\mathbb{R}^n$. By $\range(\sigma)$, we denote the range of a matrix $\sigma\in\R^{n\times m}$.
Also, with a slight abuse of notation, we shall use the symbol $|\cdot|$ both for the Euclidean norm on $\R^n$ and for the Frobenius norm on $\R^{n\times m}$. The symbol $\langle \cdot\,, \cdot \rangle$ denotes the inner product on $\R^n$. 


\section{Partial regularity of semiconvex viscosity supersolutions to HJB equations}
\label{sec:setting}

In this section, we state and prove our main result.
In the following,
given a subspace $S\subset \mathbb{R}^n$,
 and $x\in \mathcal{O}$,  
 we denote by $D^-_{S}v(x)$ (\emph{resp.,\ }
 $D_Sv(x)$) the subdifferential (\emph{resp.,\ } differential) of the restriction
 $v_{|_{(x+S)\cap \mathcal{O}}}$ of $v$ to $(x+S)\cap \mathcal{O}$, that is,
\begin{equation*}
  \begin{split}
      D^-_Sv(x)\coloneqq&  \left\{
    p\in S \colon
    \liminf_{\substack{y\rightarrow x\\y\in (x+S)\cap \mathcal{O}}}
    \frac{v(y)-v(x)-\langle p, y-x\rangle}{|y-x|}\geq 0
  \right\},\\\\
      D_Sv(x)\coloneqq&  \left\{
    p\in S\colon
    \lim_{\substack{y\rightarrow x\\y\in (x+S)\cap \mathcal{O}}}
    \frac{v(y)-v(x)-\langle p, y-x\rangle}{|y-x|}=0
  \right\}.
  \end{split}
\end{equation*}
\smallskip
   In case $S=\mathbb{R}^n$, we simply use $D^-v(x)$ (resp.\ $Dv(x)$) in place of $D^-_{\mathbb{R}^n}v(x)$ (resp.\ $D_{\mathbb{R}^n}v(x)$).
Notice that, if $D_Sv(x)$ is nonempty, then it is a singleton.
 With abuse of notation,  we shall identify a non-empty $D_Sv(x)$  with its unique element.
We denote by $P_S$ the projection of $\mathbb{R}^n$ onto $S$.
 We recall the definitions  of semiconvex function and of viscosity supersolution.
\begin{definition}[Semiconvexity]
  We say that a function $v:\mathcal{O}\to\R$ is locally semiconvex if, for each $x_{0}\in \mathcal{O}$, there exists $R>0$ with ${B}_{R}(x_{0})\subset \mathcal{O}$  and   $c_{R}\geq 0$ such that 
  $$
v(\lambda x+(1-\lambda)y)-  \lambda v(x)-(1-\lambda)v(y)\leq c_{R}\lambda(1-\lambda)|x-y|^{2}, \ \ \ \forall \lambda\in[0,1], \ \forall x,y\in B_{R}(x_{0}).
  $$ \end{definition}

\medskip
\begin{definition}[Viscosity supersolution]
  We say that  $v:\mathcal{O}\to\R$ is  a viscosity \emph{supersolution} to   HJB \eqref{2024-07-28:12} at $x\in \mathcal{O}$ if it is lower semicontinuous  at $x$ and, for every $\varphi\in C^2(\mathcal{O})$
such that $v-\varphi$ attains a 
local  minimum 
at $x$, it holds
  \begin{equation*}
    -\mathcal{H} (x,v(x),D\varphi(x),D^2\varphi(x))\geq 0.
  \end{equation*}
\end{definition}
\medskip

We now  turn to the main result.
\begin{theorem}
  \label{prop:main}
  Let $v\colon \mathcal{O} \to \R$ be a locally semiconvex supersolution at each $x\in \mathcal{O}$ to HJB \eqref{2024-07-28:12}. 

\smallskip
\begin{enumerate}[(i)]
\item  \label{PropIt1}
Let
$R (x)\coloneqq \Span\big\{\range(\sigma(x,a))\colon a\in A\big\}$ and
assume $R (x)\neq \{0\}$.
Then,
$v$ is differentiable at $x$ along the subspace $R (x)$,
and 
\begin{equation*}
  D_{R (x)}v(x)=P_{R (x)}D^-v(x).
\end{equation*}


\item \label{PropIt2}
 Let $\mathcal{D}\subset \mathcal{O}$ be
 open,
 and
  let $\{S(x)\}_{x\in \mathcal{D}}$ be a family of non-trivial subspaces of $\mathbb{R}^n$
  such that
$S(x)\subseteq R(x)$ 
for each $x\in\mathcal{D}$.
Assume that this family of subspaces  varies continuously with $x$, in the sense that the map
\begin{equation*}
 P_S\colon \mathcal{D}\rightarrow \mathcal{L}(\mathbb{R}^n), \quad x
 \mapsto P_{S(x)}
\end{equation*}
is continuous.\footnote{In particular, the map $\mathcal{D}\rightarrow \mathbb{R},\ x\mapsto \dim S(x)$, must be  constant on each connected component of $\mathcal{D}$. On the other hand, if there exists a selection $\{a_x\}_{x\in \mathcal{D}}\subseteq A$
 such that,    $\dim \range(\sigma(x,a_x))$ is non-zero and constant on each connected component of $\mathcal{D}$,  
    and such that
  $\mathcal{D}\rightarrow \mathbb{R}^{n\times m},\ x\mapsto
  \sigma(x,a_x)$, is continuous,
  then such  assumption is satisfied for $S(x)=\range(\sigma(x,a_x))$.
}
Then,
the map
\begin{equation*}
 D_Sv\colon \mathcal{D}\rightarrow \mathbb{R}^n, \quad  x \mapsto D_{S(x)}v(x)
\end{equation*}
is well-defined and continuous.
In particular, if $R(x)=\mathbb{R}^n$ for all $x\in \mathcal{ D}$, then $v\in C^1(\mathcal{ D})$.
\end{enumerate}
\end{theorem}
\begin{proof}
We prove the two items separately.
\vskip0.5em
\noindent \emph{Proof of (\ref{PropIt1}).}
We split the proof into three claims.
  \vskip0.5em
\noindent \textsc{Claim I.}
  Let $a\in {A}$ be such that $\range  \left( \sigma(x,a) \right) \neq \{0\}$, and let
$h\in \range  \left( \sigma(x,a) \right) \setminus \{0\}$.
Then $v$
is differentiable at $x$ along the direction $h$.
  \vskip0.5em

\noindent Let $ x,h$ be as in \textsc{Claim I}. Without loss of generality, up to translation, we can assume $ x=0$.
Since $v$ is locally semiconvex,
by \cite[Prop.\ 1.1.3]{CannarsaSinestrari} (equivalences (a)--(c)),
there exist a neighbourhood $\mathcal{U}\subset \mathcal{O}$ of $0$,
a convex function $v_0\colon \mathcal{U}\rightarrow \mathbb{R}$, and $\kappa\geq 0$ such that
\begin{equation}\label{2024-07-24:00}
  v(y) = v_0(y) - \frac\kappa{2}|y|^2, \quad \forall y \in \mathcal{U}.
\end{equation}
By \eqref{2024-07-24:00}, it is sufficient to show that
$v_0$ is differentiable at $0$ along the direction $h$.
Assume, by contradiction, that it is not. Then, by
\cite[Theorem~23.4]{rockafellar1970}, there exist
$p_1,p_2\in D^-v_0(0)$ such that
\begin{equation}
  \label{2024-07-24:02}
    \langle p_1,h\rangle
    <
        \langle p_2,h\rangle.
\end{equation}
By convexity of $v_0$, we have
\begin{equation}\label{2024-07-24:06}
  v_0(y)\geq  v_0(0)+\langle p_1,y\rangle\vee \langle p_2,y\rangle\qquad \forall y\in \mathcal{U}.
\end{equation}
Since
$h\in \range  \left( \sigma(0,a) \right)   \setminus \{0\}$,
 there exists $\hat h\in \mathbb{R}^m$ such that
 $h=\sigma(0,a)\hat h$.
Then
inequality
\eqref{2024-07-24:02} entails
\begin{equation}
  \label{sstarp}
  \sigma^*(0,a)(p_1-p_2)\neq 0.
\end{equation}

\noindent
For  $j\in \mathbb{N}$, let $\rho_j\colon \mathbb{R}\rightarrow \mathbb{R}$ be a function with the following properties:
\begin{subequations}
  \begin{align}
    \label{2024-07-28:01}
    &\rho_j(0)=0\\
    \label{2024-07-28:02}
    &\rho_j(t)=\rho_j(-t)\qquad \forall t\in \mathbb{R}\\
    \label{2024-07-28:03}
    &  \rho_j(t)\leq |t|\qquad\qquad \forall t\in \mathbb{R}\\
    \label{2024-07-28:04}
    &   \rho_j\in C^2(\mathbb{R})\\
    \label{2024-07-28:05}
    &   \rho_j'(0)=0\\
    \label{2024-07-28:06}
    &     \rho_j''(0)= \lambda_j\coloneqq 4\frac{j+\kappa|\sigma(0,a)|^2}{|\sigma^*(0,a)(p_1-p_2)|^2}
  \end{align}
\end{subequations}
For example, the function $\displaystyle{\rho_j(t)=-\lambda_j^3 t^4+\frac{\lambda_j}{2}t^2}$ fulfills the requirements  \eqref{2024-07-28:01}--\eqref{2024-07-28:06}.
Now, define 
\begin{equation*}
\hat\varphi_j\colon \mathbb{R}^n\rightarrow\mathbb{R}, \ \ \ \ 
  \hat\varphi_j(y)\coloneqq \rho_j
  \left(
    \frac{1}{2}
  \left| \langle p_1-p_2,y\rangle\right|
  \right)
  +
  \frac{1}{2}
  \langle p_1+p_2,y\rangle\qquad \forall y\in \mathbb{R}^n.
\end{equation*}
Then,  by \eqref{2024-07-28:02},
\begin{equation*}
    \hat\varphi_j(y)
      = 
   \rho_j\left(\frac{1}{2}
       \langle p_1-p_2,y\rangle
   \right)   +  \frac{1}{2} \langle p_1+p_2,y\rangle\qquad \forall y\in \mathbb{R}^n.
\end{equation*}
The last expression tells us that
$\hat\varphi_j\in C^2(\mathbb{R}^n)$, and
\begin{subequations}
  \begin{equation}
    \label{2024-07-28:08}
    D\hat\varphi_j(0)=\frac{1}{2}(p_1+p_2) 
  \end{equation}
  \begin{equation}
    \label{2024-07-28:09}
    \langle D^2\hat\varphi_j(0)y,z
    \rangle
    =\frac{1}{4}\rho''_j(0)
    \langle
    p_1-p_2,y
    \rangle   \langle
    p_1-p_2,z
    \rangle\qquad \forall y,z\in \mathbb{R}^n.
  \end{equation}
\end{subequations}
Moreover, for all $y\in \mathcal{U}$,
\begin{equation}\label{2024-07-28:07}
  \begin{split}
    \hat\varphi_j(y)-\frac{\kappa}{2}|y|^2&
                                                 \underset{\footnotesize{\mbox{(by \eqref{2024-07-28:03})}}}{\leq} \frac{1}{2}|\langle p_{1}-p_{2},y\rangle|+\frac{1}{2}\langle p_{1}+p_{2},y\rangle -\frac{\kappa}{2}|y|^{2}
                                                \\& \ \ \ \ 
                                                =\ \ \ \ \ 
  \langle p_1,y\rangle \vee \langle p_2,y\rangle-\frac{\kappa}{2}|y|^2\\
                                          &
   \underset{\footnotesize{\mbox{(by \eqref{2024-07-24:06})}}}{\leq}
                                            v_0(y)-\frac{\kappa}{2}|y|^2-v_0(0)\\
                                          &
           \underset{\footnotesize{\mbox{(by \eqref{2024-07-24:00})}}}{\leq}
           v(y)-v(0)  .
\end{split}
\end{equation}
Define
\begin{equation*}
  \varphi_j(y)\coloneqq \hat\varphi_j(y)-\frac{\kappa}{2}|y|^2 \qquad \forall y\in \mathbb{R}^n.
\end{equation*}
By \eqref{2024-07-28:08} and \eqref{2024-07-28:09}, we have
\begin{subequations}
  \begin{equation}
    \label{2024-07-28:10}
    D\varphi_j(0)=\frac{1}{2}(p_1+p_2) 
  \end{equation}
  \begin{equation}
    \label{2024-07-28:11}
    \langle D^2\varphi_j(0)y,z
    \rangle
    =\frac{1}{4}\rho''_j(0)
    \langle
    p_1-p_2,y
    \rangle   \langle
    p_1-p_2,z
    \rangle-\kappa\langle y,z\rangle
    \qquad \forall y,z\in \mathbb{R}^n.
  \end{equation}
\end{subequations}
By 
\eqref{2024-07-28:11},
\begin{equation}\label{2024-07-28:13}
  \begin{split}
    \Tr
\left(  \sigma^*(0,a)D^2\varphi_j(0) \sigma(0,a)\right)
=
\frac{1}{4}\rho''_j(0)
  |
  \sigma^*(0,a) (p_1-p_2)|^2
   -\kappa|\sigma(0,a)|^2
   \underset{\textrm{(recalling \eqref{2024-07-28:06})}}{=}
j.
\end{split}
\end{equation}

\noindent By \eqref{2024-07-28:07}, $v-\varphi_j$ attains a  local minimum at $0$.
Since $v$ is a viscosity supersolution at $0$ of
\eqref{2024-07-28:12},
by recalling
      \eqref{2024-07-28:10} and
      \eqref{2024-07-28:13}
we must have, for all $j\in \mathbb{N}$, 
\begin{equation}\label{2024-09-09:04}
  \begin{split}
    0 &\geq \mathcal{H}  \Big(
    0,v(0),D\varphi_j(0),D^2\varphi_j(0) \Big) \\
    &\geq
    \mathcal{L} \left(a, 0,v(0), D\varphi_j(0),  \Tr    \left(  \sigma^*(0,a)D^2\varphi_j(0) \sigma(0,a)\right) \right)  \\
    &
    =
    \mathcal{L} \left(a, 0,v(0), \frac{1}{2}(p_1+p_2),j \right)    
    \\&= g(0,a) + \frac{1}{2}\langle \beta(0,a), p_1+p_{2} \rangle + \frac{j}{2}  - \rho(0,a)v(0),
  \end{split}
\end{equation}
Letting $j\to \infty$  in \eqref{2024-09-09:04}, we obtain
\begin{equation*}
  0\geq \lim_{j\rightarrow\infty }\left(
  g(0,a) + \frac{1}{2}\langle \beta(0,a), p_1+p_{2} \rangle + \frac{j}{2}  - \rho(0,a)v(0)\right) =+\infty,
\end{equation*}
which provides the desired contradiction and completes the proof of \textsc{Claim I}.

\vskip0.5em
\noindent \textsc{Claim II.}
Let
$h\in R(x) \setminus \{0\}$.
Then $v$ is
differentiable at $x$ along the subspace $\mathbb{R}h:=\Span\{h\}$.
\vskip0.5em
\noindent
As done to prove \textsc{Claim I}, we can assume $x=0$, and
we show that 
the convex function $v_0$ introduced above is differentiable at $0$ along the subspace $\mathbb{R}h$.
  By definition of $R(0)$, we have $h=h_1+\ldots+h_k$, with $h_1\in \range(\sigma(0,a_1)),\ldots, h_k\in \range(\sigma(0,a_k))$, for some $a_1,\ldots,a_k\in A$.
  By \textsc{Claim I}, we know that $v_0$ is differentiable at $0$ along each subspace $\mathbb{R}h_j$ with differential $D_{\mathbb{R}h_j} v_0(0)$. By \cite[Theorem~25.2]{rockafellar1970} we conclude.

\vskip0.5em
\noindent \textsc{Claim III.}
The function $v$ is
differentiable at $x$ along the subspace $R(x)$, and
\begin{equation*}
  D_{R (x)}v(x)=P_{R (x)}D^-v(x).
\end{equation*}
\vskip0.5em
\noindent
Let  $v_0$ be as in the proof of \textsc{Claim I}.
By \textsc{Claim II}, we know that $v_0$ is differentiable along each $h\in R(x)$.
Then, taking into account \cite[Theorem~23.4]{rockafellar1970},
$P_{R(x)}D^-v_0(x)$  must be a singleton, and so the same holds true for
$P_{R(x)}D^-v(x)$.
On the other hand,
again due to the differentiability of
$v_0$  along each $h\in R(x)$, and
by applying \cite[Theorems~23.4 and 25.1]{rockafellar1970}
to the restriction $v_{0|R(x)}$,
we have that
$D^-_{R(x)}v_0$ is a singleton and that $v_0$ is differentiable at $x$ along $R(x)$,
and so the same holds true for $v$.
Moreover, by the very definition, we always have $D^-_{R(x)}v(x)\supseteq P_{R(x)}D^-v(x)$.
It then follows 
$D_{R(x)}v(x)=P_{R(x)}D^-v(x)$.

\vskip0.5em
\noindent \emph{Proof of \eqref{PropIt2}.}
By \eqref{PropIt1}, 
$v$ is differentiable at $x$ along $R(x)$, hence along $S(x)$, and so the map $D_Sv$ is well-defined.
Moreover, since we clearly have $D_{S(x)}v(x)=P_{S(x)}D_{R(x)}v(x)$, by
\eqref{PropIt1} we also have $D_{S(x)}v(x)=P_{S(x)}D^-v(x)$.
It remains to argue why $x\mapsto P_{S(x)}D^-v(x)$ is continuous on $\mathcal{D}$.
But this is a straightforward consequence
of \cite[Prop.\ 3.3.4(a)]{CannarsaSinestrari} and the continuity of $P_S$, after recalling that
 $D^-v$ is 
locally bounded on $\mathcal{D}$.
The final statement of (\ref{PropIt2}) is derived by setting $S(x)\coloneqq
R(x)=  \mathbb{R}^n$.
\end{proof}


\section{Applications to stochastic drift-control problems}
\label{sec:drift-control}

In this section, we illustrate some applications of our results to specific classes of stochastic optimal control problems. Let \((\Omega, \mathcal{F}, \mathbb{F} \coloneqq (\mathcal{F}_t)_{t \geq 0}, \P)\) be a complete filtered probability space, where \((W_t)_{t \geq 0}\) is an \(m\)-dimensional Brownian motion with respect to \((\mathbb{F}, \P)\).

\subsection{Partial regularity of the value function}
Let \(A\) be a Borel space, and let \(\mathcal{A}\) denote a suitable subset of \(A\)-valued \(\mathbb{F}\)-progressively measurable processes. We consider  functions \(\beta \colon \mathbb{R}^n \times A \to \mathbb{R}^n\) and \(\sigma \colon \mathbb{R}^n \to \mathbb{R}^{n \times m}\), both assumed to be measurable, and study the drift-controlled stochastic differential equation (SDE):
\begin{equation}
\label{eq:SDEdriftcontrol}
\begin{dcases}
    dX_t = \beta(X_t, \alpha_t)dt + \sigma(X_t)dW_t, \qquad t > 0, \\
    X_0 = x \in \mathbb{R}^n,
\end{dcases}
\end{equation}
where \(\alpha \in \mathcal{A}\). 
Next, consider a measurable cost function \(g \colon \mathbb{R}_+ \times \mathbb{R}^n \times A \to \mathbb{R}\) and a discount rate \(\rho \colon \mathbb{R}^n \to \mathbb{R}_+\). The value function of the  stochastic control problem is
\begin{equation*}
V(x) \coloneqq \sup_{\alpha \in \mathcal{A}} \mathcal{J}(x; \alpha), \qquad x \in \mathbb{R}^n,
\end{equation*}
where the objective functional \(\mathcal{J}:\mathbb{R}^n \times \mathcal{A} \to \mathbb{R}\) is given by
\begin{equation*}
\mathcal{J}(x; \alpha) \coloneqq \E \left[ \int_0^\infty e^{-\int_0^t \rho(X_s)ds} g(X_t, \alpha_t) dt \right], \qquad x \in \mathbb{R}^n, \ \alpha \in \mathcal{A}.
\end{equation*}

We assume that the SDE \eqref{eq:SDEdriftcontrol} admits a unique strong solution such that  the functional $\mathcal{J}(x,\alpha)$ is well-defined and finite for all \(x \in \mathbb{R}^n\) and $\alpha \in \mathcal{A}$,  and that  the value function \(V\) is also finite for all $x\in\R^{n}$. For some standard assumptions, guaranteeing the previous properties, we refer to Remark \ref{remrem}(i).

Under appropriate conditions, it can be shown (see Remark \ref{remrem}(ii)) that \(V\) is continuous and satisfies the \emph{Dynamic Programming Principle} (DPP). As a consequence, \(V\) can be typically shown to be a viscosity solution to the associated semilinear Hamilton-Jacobi-Bellman (HJB) equation:
\begin{equation*}
\mathcal{H}(x, v(x), D v(x), D^2 v(x)) = 0,
\end{equation*}
where 
\begin{equation*}
\mathcal{H}(x, r, p, P) \coloneqq \sup_{a \in A} \Big\{ g(x, a) + \langle \beta(x, a), p \rangle \Big\} + \frac{1}{2} \Tr[\sigma(x)\sigma^{*}(x)P] - \rho(x)r.
\end{equation*}
for \(x \in \mathbb{R}^n\), \(r \in \mathbb{R}\), \(p \in \mathbb{R}^n\), and \(P \in \mathcal{S}_n\).
Moreover, by direct estimates on the  functional $\mathcal{J}$, it is often possible to prove that \(V\) is locally semiconvex (see Remark \ref{remrem}(iii)). This  in turn allows us to apply our Theorem \ref{prop:main} above, yielding the partial regularity of \(V\).

\begin{remark}\label{remrem}

\begin{enumerate}[(i)]
\item  Typical requirements for the well-posedness of the control problem are the following ones: The coefficients $\beta,\sigma,g,\rho$ are Lipschitz-continuous and satisfy suitable growth conditions (notice that Lipschitz-continuity uniform over $A$ and sublinear growth in $x$ of $\beta$ and $\sigma$ guarantee existence and uniqueness of a strong solution to the controlled SDE \ref{eq:SDEdriftcontrol});  $\rho$ is sufficiently large (with respect to the growth of $g$ and the expected growth of $X$) in order to guarantee finiteness of the control functional and of the value function. We refer, e.g., to \cite[Ch.\,3]{pham}.

Clearly, those requirements are only sufficient and one may try, depending on problem's data, to treat also settings in which the controlled SDE has not Lipschitz coefficients. For instance, in dimension 1, there is a well-established theory   of existence and uniqueness of strong solutions to SDEs with non-Lipschitz data (see \cite{IkedaWatanabe}); in this regard, a relevant example in  mathematical finance is the Cox-Ingersoll-Ross process.

\item The fact that $V$ satisfies a Dynamic Programming Principle and that it is a viscosity solution to the associated HJB equation is quite standard in stochastic optimal control: The interested reader may refer to \cite[Chapters\, 3 and 4]{pham}. It is worth noticing that, in order to apply Theorem \ref{prop:main}, it is only needed that $V$ is a viscosity \emph{supersolution} to the HJB equation, a property that directly follows from the ``easy direction'' of the DPP. 

\item Semiconvexity of the value function of maximization problems in stochastic control is a well established result when the data are semiconvex. In particular, extending to the infinite time-horizon discounted case the approach followed in the proof of \cite[Prop.\ 4.5]{YongZhou}, one can show that $V$ is semiconvex if the following conditions are satisfied: (a) The discount rate $\rho$ is such that $\inf_{x\in \R^n}\rho(x) \geq \rho_o$, for some $\rho_o>0$ large enough; (b) $g$ is semiconvex in $x$, uniformly with respect to $a \in A$; (c) $\beta,\sigma$ are differentiable with respect to $x$ and such that there exists  $L>0$ such that  
$$|\beta_x (\bar x, a) - \beta_x (x, a)| + |\sigma_x (\bar x) - \sigma_x(x)|  \leq L |\bar x - x|,  \ \ \forall  x, \bar x \in \R^n,  \ a \in A.$$

\end{enumerate}
\end{remark} 
 
\subsection{Partial regularity for optimal feedback maps and regularity upgrade}
\label{sec:OC-regupgrade}

The partial regularity of the value function $V$ plays a crucial role in constructing optimal feedback (or synthesis) maps for stochastic drift-control problems. Additionally, it serves as an intermediate step toward attaining higher regularity in the solution. These aspects are heuristically illustrated below (we refer, in particular, to  \cite{fefe} for a specific example demonstrating the application of these methods).

Consider the stochastic drift-control problem discussed in the previous subsection, but now with controlled dynamics having the following separable structure of the drift:
\begin{equation*}
  \begin{dcases}
    dX_t=\beta_{1}(X_t,\alpha_t)dt+\beta_{0}(X_{t})dt +\sigma(X_t)dW_t,\qquad \forall t > 0,\\
    {\phantom{d}}X_0=x\in \mathbb{R}^n.
  \end{dcases}
\end{equation*}
The associated HJB equation reads as
\begin{equation}\label{HJBcontrol}
\mathcal{L}^{0}v(x)+ \mathcal{N}(x,Dv(x))=0,
\end{equation}
where
\begin{align*}
  \mathcal{L}^0 v(x) &\coloneqq \langle \beta_{0}(x),Dv(x)\rangle 
  + \frac{1}{2}\Tr[\sigma(x)\sigma^{*}(x)D^2v(x)] - \rho(x)v(x), \\
  \mathcal{N}(x,Dv(x)) &\coloneqq \sup_{a\in A} 
  \Big\{g(x,a) + \langle \beta_{1}(x,a),Dv(x)\rangle\Big\}.
\end{align*}

Notice that \(\mathcal{L}^0\) is a linear operator. In contrast, the structure of the nonlinear operator \(\mathcal{N}\) -- which is the one involved in the optimization -- depends only on the projection of \(Dv\) onto the subspaces $\range(\beta_{1}(x,a))$ for \(a \in A\). Precisely, interpreting 
$$
\langle \beta_{1}(x,a),Dv(x)\rangle= \langle \beta_1(x,a),D_{\range(\beta_{1}(x,a))}v(x)\rangle
$$
and letting $$S(x):= \mbox{Span} \bigcup_{a\in A}\range(\beta_{1}(x,a)),$$
we can rewrite 
$$
\mathcal{N}(x,Dv(x))= \widehat{\mathcal{N}}(x,D_{S(x)}v(x)) \coloneqq \sup_{a\in A} 
  \Big\{g(x,a) + \big\langle \beta_{1}(x,a),D_{S(x)}v(x)\big\rangle\Big\}.
$$
Now, our result provides the existence, at each $x\in\R^n$, of $D_{\range(\sigma(x))}v(x)$  for any semiconvex supersolution $v$ to \eqref{HJBcontrol} (in particular,  for the value function $V$ of the optimal control problem). Hence, if 
\begin{equation}\label{structure}
S(x)\subseteq  \range(\sigma(x)),
\end{equation}
we can give a classical sense to the nonlinear term $ \widehat{\mathcal{N}}(x,D_{S(x)}v(x)) $. This proves to be important for two key aspects.
\begin{enumerate} 
\item The candidate optimal feedback map of the stochastic control problem is then well-defined (in the classical sense) as the multivalued map 
$$
G(x):= \mbox{argmax}_{a\in A} \Big\{g(x,a)
    +
       \big\langle \beta_{1}(x,a), D_{S(x)}V(x)\big\rangle
     \Big\}.
$$
Using this map, one may then aim to establish a verification theorem within the framework of viscosity solutions to show that $G$ is an optimal feedback map (see, e.g., \cite[Ch.\,5, Sec.\,2]{YongZhou}; see also, in a deterministic setting, \cite{FeFeTo} for a finite-dimensional problem, and \cite{feta,fegogo} for an infinite-dimensional one). 
\medskip

\item One might also go further to improve the regularity. We provide only a sketch of the argument, which must be rigorously developed on a case-by-case basis. We refer to \cite{fefe} for an example. By freezing \(f(x) := D_{S(x)}V(x)\) in the HJB equation, \(V\) becomes a viscosity solution to the \emph{linear} PDE:
\begin{equation}
    \label{eq:linearHJB}
    \mathcal{L}^0 v(x)+ \widehat{\mathcal{N}}(x,f(x))=0.
\end{equation}
Now, if the operator $\mathcal{L}^0$ is hypoelliptic (see \cite{H} for the so-called H\"ormander conditions), then, under suitable assumptions on the data, we are in the position of claiming that there exists a classical solution $v\in C^{2}$ to \eqref{eq:linearHJB} on balls $B_R$, with Dirichlet boundary condition $v=V$ on $\partial B_R$. Notice that the $C^2$-property is due to the fact that H\"ormander conditions guarantee existence of a smooth transition density for the uncontrolled  It\^o-diffusion process having $\mathcal{L}^0$ as infinitesimal generator. Then, by standard results on uniqueness of viscosity solutions (see, e.g., \cite{CIL}), it is matter to identify $V$ with $v$ and thus conclude that $V\in C^{2}$. As a byproduct of this regularity of $V$, it becomes possible to prove a classical verification theorem (see, \cite[Ch.\,5, Sec.\,1]{YongZhou}), demonstrating that $G$ from item (1) above is indeed an optimal feedback map.
\end{enumerate}
  
\begin{remark}
In the discussion above, the key assumption \eqref{structure} aligns with a well-established condition for the nonlinear part of the Hamiltonian, commonly found in the literature on stochastic optimal control and semilinear HJB equations. This theoretical framework, pioneered by the seminal work of Pardoux and Peng \cite{PaPe}, offers a probabilistic representation of semilinear PDEs, analogous to the Feynman-Kac formula for linear PDEs, through systems of forward-backward SDEs.  
%
%
\end{remark}

\section{Applications to Optimal Stopping and Impulse Control}
\label{sec:OSIC}

Throughout this section, let $(\Omega, \mathcal{F}, \mathbb{F}, \P)$ be again a complete filtered probability space, with filtration $\mathbb{F}$ satisfying the usual conditions of completeness and right-continuity, on which it is defined an $\R^m$-valued Brownian motion $W:=(W_t)_{t\geq0}$ with respect to $(\mathcal{F}, \P)$. Furthermore, we shall denote by $\mathcal{T}$ the set of $\mathbb{F}$-stopping times. 

\subsection{Optimal Stopping Problems}
\label{sec:oscp}

In this section we provide an application of Theorem \ref{prop:main} to the problem of optimal stopping of a multi-dimensional It\^o-diffusion. 

Let $\beta:\R^n \to \R^n$ and $\sigma:\R^n \to \R^{n\times m}$ be measurable and such that 
the stochastic differential equation
\begin{equation}
\label{eq:SDE}
\d X_t = \beta(X_t) \d t + \sigma(X_t) \d W_t, \quad t>0, \quad X_0=x \in \R^n,
\end{equation}
admits a unique strong solution, denoted by
$(X^{x}_t)_{t\geq 0}$.
Introduce the optimal stopping problem
\begin{equation}
\label{eq:def-u}
V(x)\coloneqq  \sup_{\tau \in \mathcal{T}}\E\Big[e^{-\int_0^{\tau}\rho(X^x_s)\d s} \,g(X^{x}_{\tau})\Big], \quad x \in \R^n,
\end{equation}
for $\rho\colon \mathbb{R}^n \to \mathbb{R}_+$ measurable, and $g \in C(\R^n;\R)$.

Assuming then that the stopping functional $\E[e^{-\int_0^{\tau}\rho(X^x_s)\d s} \,g(X^{x}_{\tau})]$ is well defined for any $\tau \in \mathcal{T}$ and that $V$ is finite and locally semiconvex, one has that both claims of Theorem \ref{prop:main} apply to $V$. To see this, notice that taking $\mathcal{O}=\R^n$, and easily adapting to our stationary  discounted setting the first step of the proof of \cite[Thm.\ 7.7]{Touzi}, one has that $V$ is a viscosity supersolution to
\begin{equation}
    \label{eq:VI}
    \min\big\{-\mathcal{L}v(x), v(x) - g(x)\big\}=0, \quad x \in \R^n,
\end{equation}
where (cf.\ \eqref{2024-09-22:00} with $g\equiv 0$)
$$\mathcal{L}v(x):=\langle \beta(x), Dv(x)\rangle+\frac{1}{2} \Tr \left( \sigma(x)\sigma^*(x)D^2v(x)\right) -\rho(x)v(x).$$
Hence, $V$ is a viscosity supersolution to \eqref{2024-07-28:12}, with $\mathcal{H}=\mathcal{L}$.

\begin{remark}
    \label{rem:wellposedOS}
As already noticed, Lipschitz-continuity and sublinear growth conditions on $\beta$ and $\sigma$  are typical assumptions to guarantee existence and uniqueness of a strong solution to \eqref{eq:SDE}. Furthermore, polynomial growth condition on $g$, $\rho$ sufficiently large, and the aforementioned Lipschitz-property of $\beta,\sigma$ yield that -- due to standard moments estimates for solutions to SDEs -- stopping functional and value $V$ are well defined and actually finite. 
\end{remark}

\begin{remark}
\label{rem:easyDP}
Notice that, in the discussion above we are only using that $V$ is a viscosity supersolution to the linear part of the variational inequality \eqref{eq:VI}. The proof of this viscosity supersolution property directly derives from the ``easy direction" of the dynamic programming principle (see, e.g., the first step of the proof of \cite[Thm.\ 7.7]{Touzi}).
\end{remark}

\begin{remark}    
\label{rem:semiconvexOS}
The assumption of local semiconvexity of \(V\) must be verified on a case-by-case basis. See Remark \ref{rem:exampleOS} below for a relevant example. However, taking, e.g., $\rho$ such that $\inf_{x\in \R^n}\rho(x) \geq \rho_o$, for some $\rho_o>0$ large enough, standard estimates on solutions to stochastic differential equations employing Gr\"onwall lemma and the Burkholder-Davis-Gundy inequality (see, e.g., the estimates in the proof of \cite[Prop.\ 4.5]{YongZhou} for a finite time-horizon control problem) guarantee that if, for $\lambda\in [0,1]$, $L>0$, $M>0$, $p\geq 2$, and for every $x, \bar x \in \R^n$
  \begin{equation*}
    \begin{split}
               |h (x)| & \leq L (1+|x|),\\ 
        |h (\bar x) - h (x)|  & \leq L |\bar x - x|, \\
         h ( \lambda \bar x + (1-\lambda)x) - \lambda h(\bar x) - (1-\lambda) h ( x)   &\leq  L \lambda (1-\lambda)  |\bar{x}-x|^2,
      \end{split}
    \end{equation*}
\noindent   for $h \in \{\beta,\sigma\}$, and 
   $$
   \begin{aligned}
    |g(x)| &\leq M(1+|x|^p), \\
    |g(\bar x) - g (x) | &\leq M (1+|\bar x|^{p-1} + |x|^{p-1})|\bar x -x |, \\
   g ( \lambda \bar x + (1-\lambda)x) - \lambda g(\bar x) - (1-\lambda) g ( x)  &\leq  M \lambda (1-\lambda) (1+|\bar x|^{p-2} + |x|^{p-2}) |\bar{x}-x|^2, 
    \end{aligned}
    $$
    then, for every $x, \bar x \in \R^n$, one has
    \begin{equation*}
      \begin{split}
    |V(x)| &\leq M (1+|x|^p), \\
    |V (\bar x) - V (x) | &\leq M (1+|\bar x|^{p-1} + |x|^{p-1})|\bar x -x |, \\
    V ( \lambda \bar x + (1-\lambda)x) - \lambda V(\bar x) - (1-\lambda) V( x)  &\leq  M \lambda (1-\lambda) (1+|\bar x|^{p-2} + |x|^{p-2}) |\bar{x}-x|^2. 
  \end{split}
\end{equation*} 
\end{remark}

\begin{remark}
    \label{rem:exampleOS}
    A benchmark relevant example from Mathematical Finance under which $V$ is locally semiconvex (actually, convex) is that of an {\emph{exchange-of-baskets problem}} (see \cite{ChrSal, HuOksendal, NisRogers}) in which, for $i,j \in \{1, \dots,n\}$, for some $\mu_i\in \R$ and some $a_{ij}\geq 0$,
    $$b_i(x) = \mu_i x_i, \quad \sigma_{ij}(x)=a_{ij}x_i, \quad \rho(x)=\rho_o>\max_{i=1,\dots,n}\mu_i, \quad x\in \R^n_{+},$$
    and where
    $$g(x) = \Big(K - \sum_{i=1}^n x_i\Big)^+, \quad x \in \R^n_{+},$$
for some strike price $K>0$.
   \end{remark}

\subsection{Impulse control problems}
\label{sec:icp}

In this section, we consider the case of an impulse control problem for a multi-dimensional It\^o-diffusion process. 

Consider the set $\mathcal{A}$ of all the couples $\mathbf{a}\coloneqq \{(\tau_i,\xi_i)\}_{i \in \mathbb{N}}$, where:
\begin{itemize}
    \item[(a)] $\tau_i \in \mathcal{T}$, such that $\tau_i \leq \tau_{i+1}$ over the set $\{\tau_i<\infty\}$, for any $i\in \mathbb{N}$, and $\lim_{i\to\infty}\tau_i=\infty$ $\P$-a.s.;
    \item[(b)] $\xi_i \in \R^n$, such that $\xi_i \in \mathcal{F}_{\tau_i}$, for any $i\in \mathbb{N}$,
    \end{itemize}
    
    Let $\beta:\R^n \to \R^n$ and $\sigma:\R^n \to \R^{n\times m}$ measurable and such that, for any $\mathbf{a} \in \mathcal{A}$, there exists a unique (up to indistinguishability) $\mathbb{F}$-adapted c\`adl\`ag process $(X^{x,\mathbf{a}}_t)_{t\geq 0}$ solving in the strong sense the following SDE in integral form:
\begin{equation}
    \label{eq:SDE-IC}
    X^{x,\mathbf{a}}_t = x + \int_0^t b(X^{x,\mathbf{a}}_s) \d s + \int_0^t \sigma(X^{x,\mathbf{a}}_s) \d W_s + \sum_{i \in \mathbb{N}\colon \, \tau_i \leq t} \xi_i, \quad t \geq 0^-, \quad x \in \R^n.
\end{equation}

For measurable functions $g:\R^n \to \R_+$, $\rho:\R^n \to \R_+$, and constants $c_0>0$, $c_1>0$, introduce the optimal impulse control problem
\begin{equation}
    \label{eq:IC}
    V(x)\coloneqq \sup_{\mathbf{a} \in \mathcal{A}}\E\bigg[\int_0^{\infty}e^{-\int_0^{t}\rho(X^{x,\mathbf{a}}_s) \d s} g(X^{x,\mathbf{a}}_t) \d t - \sum_{i\in \mathbb{N}} e^{-\int_0^{\tau_i}\rho(X^{x,\mathbf{a}}_s) \d s} \big(c_0|\xi_i| + c_1\big)\bigg], \quad x \in \R^n.
\end{equation}
where the impulse control functional is given by
\begin{equation*}
\label{eq:impulsefunct}
\mathcal{J}(x;\textbf{a}):=\E\bigg[\int_0^{\infty}e^{-\int_0^{t}\rho(X^{x,\mathbf{a}}_s) \d s} g(X^{x,\mathbf{a}}_t) \d t - \sum_{i\in \mathbb{N}} e^{-\int_0^{\tau_i}\rho(X^{x,\mathbf{a}}_s) \d s} \big(c_0|\xi_i| + c_1\big)\bigg], \quad x \in \R^n,\, \textbf{a} \in A.    
\end{equation*}

If the impulse control functional is well defined and $V$ is finite and locally semiconvex, then both claims of Theorem \ref{prop:main} apply to $V$.

As a matter of fact, to see this it is enough to take $\mathcal{O}=\R^n$ and exploit the second step in the proof of \cite[Thm.\ 12.8]{OksSul} (see also \cite[Thm.\ 3.2]{Guo}), to prove that $V$ is a viscosity supersolution to 
\begin{equation}
    \label{eq:VI-impulse}
    \min\big\{-\mathcal{L}v(x) -g(x), v(x) - \mathcal{M}v(x)\big\}=0, \quad x \in \R^n,
\end{equation}
for $\mathcal{M}v(x)\coloneqq \sup_{\xi \in \R^n}\big(v(x+\xi) - c_0|\xi| - c_1\big)$
and (cf.\ \eqref{2024-09-22:00})
$$\mathcal{L}v(x):=g(x) + \langle \beta(x), Dv(x)\rangle+\frac{1}{2} \Tr \left( \sigma(x)\sigma^*(x)D^2v(x)\right) -\rho(x)v(x).$$
Hence, $V$ is a viscosity supersolution to \eqref{2024-07-28:12}, with $\mathcal{H}=\mathcal{L}$.

\begin{remark}
The impulse control functional $\mathcal{J}$
is well defined over $\R^n \times \mathcal{A}$ (but potentially infinite) if
$$\E\Big[ \sum_{i\in \mathbb{N}} e^{-\int_0^{\tau_i}\rho(X^{x,\mathbf{a}}_s) \d s} \big(|\xi_i| + 1\big)\Big] < \infty.$$
Finiteness of the value function $V$ is typically achieved by assuming suitable growth conditions on $g$, sublinear growth and Lipschitz-continuity of $\beta,\sigma$ (which, in particular, also ensure well-posedness of the controlled SDE) and taking the discount factor $\rho$ sufficiently large (to compensate the growth of $g$ and the expected trend of the controlled process). 
\end{remark}

\begin{remark}
    \label{rem:DPimpulse}
As for the previous cases of drift-control and stopping problems, in this section we are only using that $V$ is a viscosity supersolution to the linear part of the quasi-variational inequality \eqref{eq:VI-impulse}.
Again, the proof of such viscosity supersolution property directly derives from the ``easy direction" of the dynamic programming principle.
\end{remark}

\begin{remark}
    \label{rem:ICsemiconvex}
    The assumption of local semiconvexity of $V$ needs to be verified on a case by case basis. Typically, this is true if the discount rate $\rho$ is sufficiently large and if $g$ is locally semiconvex and satisfies suitable growth conditions. Precise estimates have been obtained in \cite{FeRoTa}, in a one-dimensional setting. However, it is easily seen that the approach developed in \cite{FeRoTa} (under Assumptions 2.1 and 3.3 therein) can be extended to multi-dimensional frameworks as well in order to ensure the desired local semiconvexity of $V$ as in \eqref{eq:IC}.
\end{remark}

\begin{remark}
    \label{rem:SSC}
    Statements similar to those for impulse and optimal stopping problems can be made also for a singular stochastic control problem of the form 
    \begin{equation}
    \label{eq:OC}
    V(x)\coloneqq \sup_{\mathbf{\nu} \in \mathcal{A}}\E\bigg[\int_0^{\infty}e^{-\int_0^{t}\rho(X^{x,\mathbf{\nu}}_s) \d s} g(X^{x,\mathbf{\nu}}_t) \d t - \int_{[0,\infty)} e^{-\int_0^{\tau_i}\rho(X^{x,\mathbf{\nu}}_s) \d s} c_0 \d |\mathbf{\nu}|_t\bigg], \quad x \in \R^n.
\end{equation}

\noindent Here: $\mathcal{A}$ is the set of $\mathbb{F}$-adapted, $\R^n$-valued stochastic process $\mathbb{\nu}$ with paths that are c\`adl\`ag, locally of bounded variation (componentwise), and such that $\nu_{0^-}=0$ $\P$-a.s.; $|\mathbf{\nu}|$ denotes the total variation
process induced by $\mathbb{\nu} \in \mathcal{A}$ (\,\footnote{Precisely, $|\nu|_t:=\sup\big\{\sum_{i=1}^k |\nu_{t_i} - \nu_{t_{i-1}}|:\, 0=:t_0 < t_1 < \dots < t_k:=T\big\}$.}); $(X^{x,\mathbf{\nu}}_t)_{t\geq0}$ is the unique strong solution to
\begin{equation}
    \label{eq:SDE-SSC}
    X^{x,\mathbf{\nu}}_t = x + \int_0^t b(X^{x,\mathbf{\nu}}_s) \d s + \int_0^t \sigma(X^{x,\mathbf{\nu}}_s) \d W_s + \nu_t, \quad t \geq 0^-, \quad x \in \R^n.
\end{equation}
However, differently to optimal stopping and impulse control problems, in singular stochastic control problems the $C^1$-property of the value function is typically not sufficient for a characterization of the optimal control, which in fact relies on a second-order smooth-fit property (i.e.\ the $C^2$-regularity of the value function in the direction of the controlled state variable). 

Furthermore, under suitable concavity requirements, one can show that $V$ is concave (see, e.g., \cite{DiaFe} and references therein). This, combined with the (local) semiconvexity of $V$ actually already implies $V\in C^{1,\text{Lip}}_{\text{loc}}(\R^n)$.
\end{remark}

\newpage

\section*{Acknowledgments}
Giorgio Ferrari was financed by the \emph{Deutsche Forschungsgemeinschaft} (DFG, German Research Foundation) - Project-ID 317210226 - SFB 1283. 

Salvatore Federico is a member of the GNAMPA-INdAM group and of the PRIN project 2022BEMMLZ ``Stochastic control and games and the
role of information'', financed by the Italian Ministery of University and Research. 

The authors thank Tiziano De Angelis for useful suggestions on the literature and Diego Moreira for valuable discussions.

\end{document}